\pdfoutput=1
\documentclass[11pt, a4paper, twoside,leqno]{amsart}
\usepackage[centering, totalwidth = 385pt, totalheight = 600pt]{geometry}
\usepackage{amssymb, amsmath, amsthm, thmtools, thm-restate}
\usepackage{microtype, stmaryrd, url}
\usepackage[shortlabels]{enumitem}
\usepackage[latin1]{inputenc}
\usepackage{color}
\definecolor{darkgreen}{rgb}{0,0.45,0}
\usepackage[colorlinks,citecolor=darkgreen,linkcolor=darkgreen]{hyperref}
\usepackage[british]{babel}

\DeclareMathAlphabet{\mathbf}{OT1}{cmr}{b}{n}
\makeatletter
\def\@cite#1#2{[{#1\if@tempswa ,~#2\fi}]}% NEW
\makeatother

\usepackage[arrow, matrix, tips, curve, graph, rotate]{xy}
\SelectTips{cm}{10}

\makeatletter
\def\matrixobject@{%
  \edef \next@{={\DirectionfromtheDirection@ }}%
  \expandafter \toks@ \next@ \plainxy@
  \let\xy@@ix@=\xyq@@toksix@
  \xyFN@ \OBJECT@}
\let\xy@entry@@norm=\entry@@norm
\def\entry@@norm@patched{%
  \let\object@=\matrixobject@
  \xy@entry@@norm }
\AtBeginDocument{\let\entry@@norm\entry@@norm@patched}
\makeatother

\newcommand{\twocong}[2][0.5]{\ar@{}[#2] \save ?(#1)*{\cong}\restore}
\newcommand{\twoeq}[2][0.5]{\ar@{}[#2] \save ?(#1)*{=}\restore}
\newcommand{\rtwocell}[3][0.5]{\ar@{}[#2] \ar@{=>}?(#1)+/l 0.2cm/;?(#1)+/r 0.2cm/^{#3}}
\newcommand{\ltwocell}[3][0.5]{\ar@{}[#2] \ar@{=>}?(#1)+/r 0.2cm/;?(#1)+/l 0.2cm/^{#3}}
\newcommand{\ltwocello}[3][0.5]{\ar@{}[#2] \ar@{=>}?(#1)+/r 0.2cm/;?(#1)+/l 0.2cm/_{#3}}
\newcommand{\dtwocell}[3][0.5]{\ar@{}[#2] \ar@{=>}?(#1)+/u  0.2cm/;?(#1)+/d 0.2cm/^{#3}}
\newcommand{\dltwocell}[3][0.5]{\ar@{}[#2] \ar@{=>}?(#1)+/ur  0.2cm/;?(#1)+/dl 0.2cm/^{#3}}
\newcommand{\drtwocell}[3][0.5]{\ar@{}[#2] \ar@{=>}?(#1)+/ul  0.2cm/;?(#1)+/dr 0.2cm/^{#3}}
\newcommand{\dthreecell}[3][0.5]{\ar@{}[#2] \ar@3{->}?(#1)+/u  0.2cm/;?(#1)+/d 0.2cm/^{#3}}
\newcommand{\utwocell}[3][0.5]{\ar@{}[#2] \ar@{=>}?(#1)+/d 0.2cm/;?(#1)+/u 0.2cm/_{#3}}
\newcommand{\dtwocelltarg}[3][0.5]{\ar@{}#2 \ar@{=>}?(#1)+/u  0.2cm/;?(#1)+/d 0.2cm/^{#3}}
\newcommand{\utwocelltarg}[3][0.5]{\ar@{}#2 \ar@{=>}?(#1)+/d  0.2cm/;?(#1)+/u 0.2cm/_{#3}}

\newcommand{\pushoutcorner}[1][dr]{\save*!/#1+1.2pc/#1:(1,-1)@^{|-}\restore}
\newcommand{\pullbackcorner}[1][dr]{\save*!/#1-1.2pc/#1:(-1,1)@^{|-}\restore}

\newdir{(}{{}*!<0em,-.14em>-\cir<.14em>{l^r}}
\newdir{ (}{{}*!/-5pt/\dir{(}}
\newdir{ >}{{}*!/-5pt/\dir{>}}

\DeclareMathOperator{\ob}{ob}

\newcommand{\cat}[1]{\mathbf{#1}}

\newcommand{\thg}{{\mathord{\text{--}}}}

\newcommand{\cd}[2][]{\vcenter{\hbox{\xymatrix#1{#2}}}}

\renewcommand{\phi}{\varphi}
\newcommand{\A}{{\mathcal A}}
\newcommand{\B}{{\mathcal B}}
\newcommand{\C}{{\mathcal C}}
\newcommand{\D}{{\mathcal D}}
\newcommand{\E}{{\mathcal E}}
\newcommand{\F}{{\mathcal F}}

\newcommand{\K}{{\mathcal K}}
\renewcommand{\L}{{\mathcal L}}

\renewcommand{\O}{{\mathcal O}}

\newcommand{\R}{{\mathcal R}}
\let\sec=\S
\renewcommand{\S}{{\mathcal S}}
\newcommand{\T}{{\mathcal T}}

\newcommand{\V}{{\mathcal V}}
\newcommand{\W}{{\mathcal W}}

\newcommand{\ho}{\mathrel{\bot_h}}

\newcommand{\xtor}[1]{\cdl[@1]{{} \ar[r]|-{\object@{|}}^{#1} & {}}}

\makeatletter

\def\hookleftarrowfill@{\arrowfill@\leftarrow\relbar{\relbar\joinrel\rhook}}
\def\twoheadleftarrowfill@{\arrowfill@\twoheadleftarrow\relbar\relbar}
\def\leftbararrowfill@{\arrowdoublefill@{\leftarrow\mkern-5mu}\relbar\mapstochar\relbar\relbar}
\def\Leftbararrowfill@{\arrowdoublefill@{\Leftarrow\mkern-2mu}\Relbar\Mapstochar\Relbar\Relbar}
\def\leftringarrowfill@{\arrowdoublefill@{\leftarrow\mkern-3mu}\relbar{\mkern-3mu\circ\mkern-2mu}\relbar\relbar}
\def\lefttriarrowfill@{\arrowfill@{\mathrel\triangleleft\mkern0.5mu\joinrel\relbar}\relbar\relbar}
\def\Lefttriarrowfill@{\arrowfill@{\mathrel\triangleleft\mkern1mu\joinrel\Relbar}\Relbar\Relbar}

\def\hookrightarrowfill@{\arrowfill@{\lhook\joinrel\relbar}\relbar\rightarrow}
\def\twoheadrightarrowfill@{\arrowfill@\relbar\relbar\twoheadrightarrow}
\def\rightbararrowfill@{\arrowdoublefill@{\relbar\mkern-0.5mu}\relbar\mapstochar\relbar\rightarrow}
\def\Rightbararrowfill@{\arrowdoublefill@{\Relbar\mkern-2mu}\Relbar\Mapstochar\Relbar\Rightarrow}
\def\rightringarrowfill@{\arrowdoublefill@\relbar\relbar{\mkern-2mu\circ\mkern-3mu}\relbar{\mkern-3mu\rightarrow}}
\def\righttriarrowfill@{\arrowfill@\relbar\relbar{\relbar\joinrel\mkern0.5mu\mathrel\triangleright}}
\def\Righttriarrowfill@{\arrowfill@\Relbar\Relbar{\Relbar\joinrel\mkern1mu\mathrel\triangleright}}

\def\leftrightarrowfill@{\arrowfill@\leftarrow\relbar\rightarrow}
\def\mapstofill@{\arrowfill@{\mapstochar\relbar}\relbar\rightarrow}

\renewcommand*\xleftarrow[2][]{\ext@arrow 20{20}0\leftarrowfill@{#1}{#2}}
\providecommand*\xLeftarrow[2][]{\ext@arrow 60{22}0{\Leftarrowfill@}{#1}{#2}}
\providecommand*\xhookleftarrow[2][]{\ext@arrow 10{20}0\hookleftarrowfill@{#1}{#2}}
\providecommand*\xtwoheadleftarrow[2][]{\ext@arrow 60{20}0\twoheadleftarrowfill@{#1}{#2}}
\providecommand*\xleftbararrow[2][]{\ext@arrow 10{22}0\leftbararrowfill@{#1}{#2}}
\providecommand*\xLeftbararrow[2][]{\ext@arrow 50{24}0\Leftbararrowfill@{#1}{#2}}
\providecommand*\xleftringarrow[2][]{\ext@arrow 10{26}0\leftringarrowfill@{#1}{#2}}
\providecommand*\xlefttriarrow[2][]{\ext@arrow 80{24}0\lefttriarrowfill@{#1}{#2}}
\providecommand*\xLefttriarrow[2][]{\ext@arrow 80{24}0\Lefttriarrowfill@{#1}{#2}}

\renewcommand*\xrightarrow[2][]{\ext@arrow 01{20}0\rightarrowfill@{#1}{#2}}
\providecommand*\xRightarrow[2][]{\ext@arrow 04{22}0{\Rightarrowfill@}{#1}{#2}}
\providecommand*\xhookrightarrow[2][]{\ext@arrow 00{20}0\hookrightarrowfill@{#1}{#2}}
\providecommand*\xtwoheadrightarrow[2][]{\ext@arrow 03{20}0\twoheadrightarrowfill@{#1}{#2}}
\providecommand*\xrightbararrow[2][]{\ext@arrow 01{22}0\rightbararrowfill@{#1}{#2}}
\providecommand*\xRightbararrow[2][]{\ext@arrow 04{24}0\Rightbararrowfill@{#1}{#2}}
\providecommand*\xrightringarrow[2][]{\ext@arrow 01{26}0\rightringarrowfill@{#1}{#2}}
\providecommand*\xrighttriarrow[2][]{\ext@arrow 07{24}0\righttriarrowfill@{#1}{#2}}
\providecommand*\xRighttriarrow[2][]{\ext@arrow 07{24}0\Righttriarrowfill@{#1}{#2}}

\providecommand*\xmapsto[2][]{\ext@arrow 01{20}0\mapstofill@{#1}{#2}}
\providecommand*\xleftrightarrow[2][]{\ext@arrow 10{22}0\leftrightarrowfill@{#1}{#2}}
\providecommand*\xLeftrightarrow[2][]{\ext@arrow 10{27}0{\Leftrightarrowfill@}{#1}{#2}}

\makeatother

\numberwithin{equation}{section}

\theoremstyle{plain}
\newtheorem{Thm}{Theorem}
\newtheorem*{Thm*}{Theorem}
\newtheorem{Prop}[Thm]{Proposition}

\newtheorem{Lemma}[Thm]{Lemma}

\theoremstyle{definition}
\newtheorem{Defn}[Thm]{Definition}

\newtheorem{Ex}[Thm]{Example}

\newtheorem{Rk}[Thm]{Remark}

\begin{document}
\leftmargini=2em \title[Bousfield (co)localisation of one-dimensional
model structures]{Bousfield localisation and colocalisation of
  one-dimensional model structures}

\author{Scott Balchin}
\address{Department of Pure Mathematics, The Hicks Building,
  University of Sheffield, Sheffield S3 7RH, United Kingdom}
\email{scott.balchin@sheffield.ac.uk}

\author{Richard Garner} \address{Centre of Australian Category Theory,
  Macquarie University, NSW 2109, Australia}
\email{richard.garner@mq.edu.au}

\date{\today}

\thanks{The work described here was carried out during a visit by the
  first author to Sydney supported by Macquarie University Research
  Centre funding; both authors express their gratitude for this
  support. The second author also acknowledges, with equal gratitude,
  the support of Australian Research Council grants DP160101519 and
  FT160100393. }

\begin{abstract}
  We give an account of Bousfield localisation and colocalisation for
  \emph{one-dimensional} model categories---ones enriched over the
  model category of $0$-types. A distinguishing feature of our
  treatment is that it builds localisations and colocalisations using
  only the constructions of \emph{projective} and \emph{injective}
  transfer of model structures along right and left adjoint functors,
  and without any reference to Smith's theorem.
\end{abstract}
\maketitle

\section{Introduction}
\label{sec:introduction}

\looseness=-1 A \emph{(Bousfield) localisation} of a model category
$\E$ is a model structure $\E_\ell$ on the same underlying category
with the same cofibrations, but a larger class of weak equivalences.
If $\E$ is left proper and combinatorial, one may construct a
localisation from any set $S$ of maps which one wishes to become weak
equivalences in $\E_\ell$; the fibrant objects of $\E_\ell$ will be
the \emph{$S$-local} fibrant objects of $\E$---those which see each
map in $S$ as a weak equivalence---and the weak equivalences of
$\E_\ell$, the \emph{$S$-local equivalences}---those which every
$S$-local fibrant object sees as a weak equivalence. The $S$-local
equivalences and the original cofibrations determine the other classes
of the $\E_\ell$-model structure; the hard part is exhibiting the
needed factorisations, which is usually done using a subtle
cardinality argument of
Smith~\cite[Theorem~1.7]{Beke2000Sheafifiable}.

This paper is the first step towards understanding localisations of
combinatorial\label{sec:introduction-1} model categories in a way
which avoids Smith's theorem, and instead uses only the constructions
of \emph{projective} and \emph{injective} liftings of model
structures---that is, transfers along right and left adjoint functors.
It is only a first step since, for reasons to be made clear soon, we
only implement our idea here for the rather special class of
\emph{one-dimensional} model categories: those which are enriched over
the cartesian model category of $0$-types. While homotopically
trivial, there are mathematically interesting examples of such model
structures, and in this context, our approach yields the following
complete characterisation:
\begin{restatable*}{Thm}{firstmaintheorem}
  \label{thm:3}
  If $\E$ is a left proper one-dimensional combinatorial model
  category, then the assignation $\E_\ell \mapsto (\E_\ell)_{cf}$
  yields an order-reversing bijection between combinatorial
  localisations of $\E$ (ordered by inclusion of acyclic
  cofibrations) and full, replete, reflective, locally presentable subcategories of
  $\E_{cf}$ (ordered by inclusion).
\end{restatable*}
Here, $(\thg)_{cf}$ is the operation assigning to a model category its
subcategory of cofibrant--fibrant objects. Since our approach relies
only on injective and projective liftings, it dualises
straightforwardly, giving the corresponding:
\begin{restatable*}{Thm}{secondmaintheorem}
  \label{thm:4}
  If $\E$ is a right proper one-dimensional combinatorial model
  category, then the assignation $\E_r \mapsto (\E_r)_{cf}$ yields an
  order-reversing bijection between combinatorial colocalisations of
  $\E$ (ordered by inclusion of acyclic fibrations) and full, replete,
  coreflective, locally presentable subcategories of $\E_{cf}$
  (ordered by inclusion).
\end{restatable*}

These results expand on the inquiry of~\cite{Salch2017The-Bousfield},
which characterises (co)localisations of \emph{discrete} model
categories: ones whose weak equivalences are the isomorphisms.
However, it is our general approach to constructing (co)localisations,
rather than the applications to the one-dimensional setting, which is
the main conceptual contribution of this paper, and it therefore seems
appropriate to now sketch this approach in the context of a general
combinatorial model category $\E$.

As model structures are determined by their cofibrations and their
fibrant objects, a localisation of $\E$ can be determined by
specifying its fibrant objects. So suppose given a class of fibrant
objects in $\E$, which we call \emph{local}, that we would like to
form the fibrant objects of a localisation; for example, given a set
$S$ of maps in $\E$, we could take ``local'' to mean ``$S$-local
fibrant''. We will construct the localisation at issue with reference
to an adjunction
\begin{equation}\label{eq:9}
  \cd{
    {\L} \ar@<-4.5pt>[r]_-{G} \ar@{}[r]|-{\bot} &
    {\E} \ar@<-4.5pt>[l]_-{F}
  }
\end{equation}
between $\E$ and a suitably-defined category of local objects $\L$.
Naively, we might try taking $\L$ to be the full subcategory of $\E$
on the local objects; but since this subcategory is not typically
complete nor cocomplete, its inclusion functor into $\E$ will
typically not have the required left adjoint. So instead, we take
$\L$-objects to be $\E$-objects \emph{endowed} with algebraic
structure witnessing their locality, and take $\L$-maps to be
$\E$-maps which \emph{strictly} preserve this structure. This
algebraicity of the definition of $\L$ now ensures that it is a
locally presentable category, and that the forgetful functor to $\E$
has the desired left adjoint; this
extends~\cite{Nikolaus2011Algebraic}'s construction of an adjunction
with \emph{algebraically fibrant} objects. Note that there can be many
different ways of choosing the algebraic structure which witnesses
locality, and not all of these are appropriate; indeed, choosing the
correct definition of $\L$ is the most subtle point in our argument.

Thereafter, the remainder of the argument is conceptually clear. We
first \emph{projectively} transfer the given model structure on $\E$
along the right adjoint $G \colon \L \rightarrow \E$, and then
\emph{injectively} transfer back along $F \colon \E \rightarrow \L$.
Local presentability ensures that these transfers exist so long as the
requisite acyclicity conditions are satisfied
(cf.~Proposition~\ref{prop:1} below). For the transfer to $\L$, we
verify acylicity using a path object argument, since every object of
$\L$ will be fibrant; for the transfer back to $\E$, acyclicity will
be immediate so long as $GF$ preserves weak equivalences---which might
be verified, for example, using left properness of $\E$.

At this point, we have a new model structure $\E'$ on the underlying
category of $\E$, which has more weak equivalences and cofibrations,
and makes every local object fibrant. However, it is not yet a
localisation of $\E$ since the cofibrations need not be the same.
Thus, the final step is to note that, since $\C_\E \subseteq \C_{\E'}$
and $\W_\E \subseteq \W_{\E'}$, we can use~\cite{Cole2006Mixing} to
\emph{mix} the model structures $\E$ and $\E'$, obtaining a model
structure $\E_\ell$ whose cofibrations are those of $\E$ and in which
every local object is fibrant; under appropriate homotopical closure
conditions on the class of local objects, the $\E_\ell$-fibrant
objects will be \emph{precisely} the local ones.

In this way, we may construct localisations using only the tools of
projective and injective liftings, and of mixing of model structures.
It turns out (cf.~Proposition~\ref{prop:11} below) that mixing of
model structures may be reduced in turn to liftings, so that we have a
construction of localisations from projective and injective liftings
alone. Note that this approach does \emph{not} avoid the cardinality
arguments involved in Smith's theorem; rather, it pushes them
elsewhere, namely into the construction of injective liftings of model
structures as detailed in~\cite{Makkai2014Cellular}. In particular,
our approach gives no more of an explicit grasp on the classes of maps
of a localisation than the usual one. However, we believe there are
still good reasons for adopting it.

One advantage of our approach is that it dualises
trivially to give a construction of Bousfield \emph{co}localisations,
wherein one enlarges the class of weak equivalences while fixing the
class of \emph{fibrations}; this time, one starts from the
\emph{co}local objects---those which should be the cofibrant objects
of the colocalised model structure---and constructs the desired
colocalisation with reference to an adjunction between $\E$ and a
category of ``algebraically colocal cofibrant objects''.

Another positive consequence of our approach, and our original
motivation for developing it, is that allows for an account of
(co)localisation for the \emph{algebraic} model structures of
Riehl~\cite{Riehl2011Algebraic}. These are combinatorially rich
presentations of model categories in which, among other things,
(acyclic) fibrant replacement constitutes a monad on the category of
arrows, and (acyclic) cofibrant replacement a comonad; they have been
used to derive non-trivial homotopical
results~\cite{Ching2014Coalgebraic, Barthel2013On-the-construction,
  Blumberg2014Homotopical}, and are of some importance in the
\emph{homotopy type theory} project~\cite{HoTT2013}. However, there is
no account of localisation for algebraic model structures as there
seems to be no ``algebraic'' version of Smith's theorem. On the other
hand, there \emph{are} algebraic versions of injective and projective
lifting~\cite[\sec 4.5]{Bourke2014AWFS1}; whence our interest. A
potential application of this would be to the study of localisation
for model structures which, while not cofibrantly generated in the
classical sense, \emph{are} cofibrantly generated in the algebraic
sense; see the discussion in~\cite{Bayeh2015Left-induced}.

As noted above, the subtlest point in our approach lies in choosing
the algebraic structure which constitutes the notion of
``algebraically local object''. The key issue is whether one can
construct the required path objects in $\L$, and this is sensitive
both to the choice of $\L$ and the nature of the model category $\E$;
see~\cite{Nikolaus2011Algebraic, Ching2014Coalgebraic} for some
discussion of this point. This delicacy is somewhat orthogonal to the
main thrust of our argument, and so in this paper, we sidestep it
entirely by concentrating on the situation in which the
\emph{property} and the \emph{structure} of locality necessarily
coincide. This is the setting of one-dimensional model structures, and
this is why we concentrate on this seemingly degenerate case.

In elementary terms, a model structure is one-dimensional when the
liftings involved in its factorisations are \emph{unique}. Such model
categories were introduced and investigated in~\cite{Pultr2002Free};
however, it was left open as to whether examples of such model
structures arise in mathematical practice. A subsidiary objective of
this paper is to show that, in fact, this is the case: for example, if
$A$ is a commutative ring, then there is a model structure on the
category $[\smash{\cat{Alg}_{A}^\mathrm{fp}}, \cat{Set}]$ of diagrams
of finitely presented $A$-algebras whose fibrant objects are sheaves
on the big Zariski topos of $A$ (i.e., generalised algebraic spaces
over $\mathrm{Spec}\,A$), and whose cofibrant--fibrant objects are
sheaves on the topological space $\mathrm{Spec}\,A$.

We conclude this introduction with a short overview of the contents of
the paper. In Section~\ref{sec:model-categ-backgr} we recall the
necessary model-categorical background on combinatoriality, lifting,
and mixing of model structures. In
Section~\ref{sec:1-dimensional-model}, we introduce one-dimensional
model structures and study their homotopical properties. Then in
Section~\ref{sec:local-coloc-one}, we implement our general approach
to localisation in the context of one-dimensional model structures, by
providing a set of conditions which perfectly characterise the
categories of fibrant objects in a localisation of a one-dimensional
model structure. In Section~\ref{sec:left-right-prop}, we explain how
matters are simplified by the assumption of left properness,
culminating in our first main result, Theorem~\ref{thm:3}; then in
Section~\ref{sec:relat-enrich-bousf} use this to recover the classical
account of localisation at a given set of maps in a left proper
one-dimensional model structure. In Section~\ref{sec:colocalisation},
we dualise our theory to the case of \emph{co}localisation for
one-dimensional model structures, obtaining our second main
Theorem~\ref{thm:4}; and finally, in Section~\ref{sec:examples}, we
illustrate our results with a range of examples of one-dimensional
model structures.

\section{Model-categorical background}
\label{sec:model-categ-backgr}

Throughout the paper, we write $(\C, \W, \F)$ for a model structure
with cofibrations $\C$, weak equivalences $\W$ and fibrations $\F$,
and write $\T\C = \C \cap \W$ and $\T\F = \F \cap \W$ for the acyclic
cofibrations and fibrations. We assume our model categories to be
locally small, complete and cocomplete, and endowed with functorial
factorisations; these induce functorial fibrant and cofibrant
replacements, which we write as $\eta \colon 1 \Rightarrow R$ and
$\varepsilon \colon Q \Rightarrow 1$. We write $\mathrm{RLP}(\K)$ or
$\mathrm{LLP}(\K)$ for the class of maps with the right or left
lifting property with respect to a class of maps $\K$, and write
$U^{-1}(\K)$ for the inverse image of the class under a functor $U$.

\begin{Defn}
  \label{def:6}
  Suppose that $\E$ is a category equipped with a model structure
  $(\C, \W, \F)$ and that $U \colon \D \rightarrow \E$. 
  \begin{itemize}[itemsep=0.25\baselineskip]
  \item The \emph{projectively lifted} model structure on $\D$, if it
    exists, is the one whose weak equivalences and fibrations are
    given by $U^{-1}(\W)$ and $U^{-1}(\F)$ respectively.
  \item The \emph{injectively lifted} model structure on $\D$, if it
    exists, is that whose cofibrations and weak equivalences are given
    by $U^{-1}(\C)$ and $U^{-1}(\W)$ respectively.
  \end{itemize}
\end{Defn}

The basic setting in which lifted model structures are guaranteed to
exist is that of combinatorial model categories. Recall that a model
category is called \emph{combinatorial} if its underlying category is
locally presentable~\cite{Gabriel1971Lokal}, and its two weak
factorisation systems $(\C, \T\F)$ and $(\T\C, \F)$ are cofibrantly
generated.

\begin{Prop}
  \label{prop:1}
  Let $\E$ be a combinatorial model category, let $\D$ be a locally
  presentable category, and let $U \colon \D \rightarrow \E$.
  \begin{enumerate}[(i),itemsep=0.25\baselineskip]
  \item If $U$ is a right adjoint, and the acyclicity condition
    $\mathrm{LLP}(U^{-1}(\F)) \subset U^{-1}(\W)$ holds, then the
    projective lifting along $U$ exists and is combinatorial.
  \item If $U$ is a left adjoint, and the acyclicity condition
    $\mathrm{RLP}(U^{-1}(\C)) \subset U^{-1}(\W)$ holds, then the
    injective lifting along $U$ exists and is combinatorial.
  \end{enumerate}
\end{Prop}
\begin{proof}
  (i) follows from~\cite[Theorem~11.3.2]{Hirschhorn2003Model} plus the
  fact that any set of maps in a locally presentable category permits
  the small object argument; the argument for (ii) is due
  to~\cite{Makkai2014Cellular}, but is given in the form we need
  in~\cite[Theorem~2.23]{Bayeh2015Left-induced}.
\end{proof}
Despite their surface similarity, the two parts of this result are
sharply different from each other. In (i), we obtain explicit choices
of generating (acyclic) cofibrations for $\D$ by applying $F$ to the
corresponding generators for $\E$. In (ii), by contrast, it is
typically impossible to write down explicit sets of generating
(acyclic) cofibrations for $\D$; one merely knows that they exist.

Note also the following result, which will be useful in the sequel. In
its statement, an \emph{accessible} functor is one preserving
$\kappa$-filtered colimits for a regular cardinal $\kappa$.
\begin{Prop}
  \label{prop:20}
  If $\E$ is a combinatorial model category, then it admits a
  cofibrant replacement functor $Q$ and fibrant replacement functor
  $R$ which are accessible.
\end{Prop}
\begin{proof}
  See~\cite[Proposition~2.3]{Dugger2001Combinatorial}.
\end{proof}
The use we make of this fact is encapsulated in the following standard
result from the theory of locally presentable categories.
\begin{Prop}
  \label{prop:28}
  If $\A \subseteq \E$ is a full reflective (resp.,~coreflective)
  subcategory and $\E$ is locally presentable, then $\A$ is locally
  presentable if and only if the reflector
  $R \colon \E \rightarrow \E$ (resp.,~coreflector $Q \colon \E
  \rightarrow \E)$ is accessible.
\end{Prop}
\begin{proof}
  The ``only if'' direction follows on observing that,
  by~\cite[Satz~14.6]{Gabriel1971Lokal}, any adjunction between
  locally presentable categories induces both an accessible monad and an
  accessible comonad. In the ``if'' direction, note that, in either
  case, the subcategory $\A$ is complete and cocomplete, and so
  by~\cite[Theorem~2.47]{Adamek1994Locally} will be locally
  presentable so as long as it is an accessible
  category~\cite{Makkai1989Accessible}. But $\A$ is the universal
  subcategory of $\E$ on which $\eta \colon 1 \Rightarrow R$
  (resp.,~$\varepsilon \colon Q \Rightarrow 1$) becomes invertible, and
  so by~\cite[Theorem~5.1.6]{Makkai1989Accessible} is accessible
  since $R$ (resp.,~$Q$) is so.
\end{proof}

We now recall Cole's result~\cite{Cole2006Mixing} on \emph{mixing}
model structures.

\begin{Prop}
  \label{prop:11}
  Let $(\C_1, \W_1, \F_1)$ and $(\C_2, \W_2, \F_2)$ be combinatorial
  model structures on the same category $\E$. If $\F_1 \subseteq \F_2$
  and $\W_1 \subseteq \W_2$, then there is a combinatorial \emph{mixed
    model structure} $(\C_m, \W_m, \F_m)$ on $\E$ with $\F_m = \F_1$
  and $\W_m = \W_2$.
\end{Prop}
\begin{proof}
  Consider the combinatorial model structure $(\C, \W, \F)$ on
  $\E \times \E$ which in its first component is given by
  $(\T\C_1, \mathrm{all}, \F_1)$ and in its second by
  $(\C_2, \W_2, \F_2)$. The diagonal
  $\Delta \colon \E \rightarrow \E \times \E$ is a right adjoint
  between locally presentable categories, and we have that
  $ \Delta^{-1}(\F) = \F_1 \cap \F_2 = \F_1$ and
  $\Delta^{-1}(\W) = \mathrm{all} \cap \W_2 = \W_2$. Thus, since
  $\mathrm{LLP}(\Delta^{-1}(\F)) = \T\C_1 \subseteq \W_1 \subseteq
  \W_2 = \Delta^{-1}(\W)$, the projectively lifted model structure
  $(\C_m, \W_m, \F_m)$ exists, and has $\W_m = \W_2$ and
  $\F_m = \F_1$.
\end{proof}
The proof we give is less explicit than Cole's; he constructs the
required factorisations directly rather than appealing to a lifting
result. This allows him to avoid the assumption of combinatoriality of
the two starting model structures, but means that he does not derive
it for the mixed model structure either.

It is in fact easy to derive explicit generating sets for the mixed
model structure from ones for the two given model structures. Since
$\F_m = \F_1$, a generating set of mixed acyclic cofibrations is given
by any generating set for $\T\C_1$; and since
$\T\F_m = \F_1 \cap \W_2 = \F_1 \cap \F_2 \cap \W_2 = \F_1 \cap
\T\F_2$, a generating set of mixed cofibrations is given by the union
of any generating set for $\T\C_1$ and any generating set for $\C_2$.

Cole's construction of the mixed model structure dualises without
difficulty. Our proof also dualises by using injective rather than
projective liftings. Once again, we must add the assumption of
combinatoriality of the input model structures, but gain it on the
output side. This time we cannot, in general, find explicit generating
sets of maps for the mixed model structure.

\begin{Prop}
  \label{prop:19}
  Let $(\C_1, \W_1, \F_1)$ and $(\C_2, \W_2, \F_2)$ be combinatorial
  model structures on the same category $\E$. If $\C_1 \subseteq \C_2$
  and $\W_1 \subseteq \W_2$, then there is a combinatorial
  \emph{mixed model structure} $(\C_m, \W_m, \F_m)$ on $\E$ with
  $\C_m = \C_1$ and $\W_m = \W_2$.
\end{Prop}

\section{One-dimensional model structures}
\label{sec:1-dimensional-model}

\begin{Defn}
  \label{def:5}
  The \emph{model category of $0$-types} is the category of sets
  endowed with the cartesian monoidal model structure
  $(\text{all}, \text{iso}, \text{all})$. A model category $\E$ is
  called \emph{one-dimensional} if it is enriched over the model
  category of $0$-types.
\end{Defn}

The following result characterises the underlying weak factorisation
systems of one-dimensional model structures; for a yet more
comprehensive list of characterisations,
see~\cite[Proposition~2.3]{Rosicky2007Factorization}.

\begin{Prop}
  \label{prop:17}
  The following are equivalent for a weak factorisation system $(\L,
  \R)$ on a finitely complete and cocomplete category:
  \begin{enumerate}[(i)]
  \item Every $\L$-map has the unique lifting property
    against each $\R$-map;
  \item If $f \colon A \rightarrow B$ is in $\L$, then so is the
    codiagonal $\nabla \colon B +_A B \rightarrow B$;
  \item If $f \colon A \rightarrow B$ is in $\R$, then so is the
    diagonal $\Delta \colon A \rightarrow A \times_B A$;
  \item If $f \colon A \rightarrow B$ and $g \colon B \rightarrow C$
    and $f \in \L$, then $g \in \L$ iff $gf \in \L$;
  \item If $f \colon A \rightarrow B$ and $g \colon B \rightarrow C$
    and $g \in \R$, then $f \in \R$ iff $gf \in \R$.
  \end{enumerate}
\end{Prop}
\begin{proof}
  (i) $\Leftrightarrow$ (ii) $\Leftrightarrow$ (iii) by~\cite[\sec
  4.5]{Bousfield1977Constructions}, while (i) $\Leftrightarrow$ (iv)
  $\Leftrightarrow$ (v)
  by~\cite[Satz~3]{Ringel1970Diagonalisierungspaare.}.
\end{proof}

We call a weak factorisation system satisfying these conditions
\emph{orthogonal}.

\begin{Prop}
  \label{prop:21}
  The following are equivalent for a locally small model category $\E$:
  \begin{enumerate}[(i)]
  \item $\E$ is one-dimensional;
  \item The weak factorisation systems $(\T\C, \F)$ and $(\C, \T\F)$
    of $\E$ are orthogonal.
  \end{enumerate}
\end{Prop}
\begin{proof}
  The model structure for $0$-types on $\cat{Set}$ has no generating
  acyclic cofibrations, and generating cofibrations
  $\{0 \rightarrow 1, \ 2 \rightarrow 1\}$. For any map
  $f \colon A \rightarrow B$ in $\E$, its pushout tensor with
  $0 \rightarrow 1$ in $\cat{Set}$ is $f$, while its pushout tensor
  with $2 \rightarrow 1$ is $\nabla \colon B +_A B \rightarrow B$. So
  $\E$ is enriched over the model structure for $0$-types if and only
  if both $\C$ and $\T\C$ satisfy the closure condition in
  Proposition~\ref{prop:17}(ii).
\end{proof}

By the standard properties of orthogonal factorisation
systems~\cite{Freyd1972Categories}, factorisations of maps in a
one-dimensional model category $\E$ are unique to within unique
isomorphism. We also have the following good behaviour of the full
subcategories $i \colon \E_f \hookrightarrow \E$ and
$j \colon \E_c \hookrightarrow \E$ of fibrant and cofibrant objects:
\begin{Prop}
  \label{prop:22}
  For any one-dimensional model category $\E$, there are adjunctions
  \begin{equation*}
    \cd{
      {\E_f} \ar@<-4.5pt>[r]_-{i} \ar@{}[r]|-{\bot} &
      {\E} \ar@<-4.5pt>[l]_-{R} & & 
      {\E_c} \ar@<-4.5pt>[r]_-{j} \ar@{}[r]|-{\top} &
      {\E\rlap{ .}} \ar@<-4.5pt>[l]_-{Q}
    }
  \end{equation*}
  In particular, both $\E_c$ and $\E_f$ are complete and cocomplete; if
  $\E$ is combinatorial, then they are moreover locally presentable.
\end{Prop}
\begin{proof}
  Each $Y \in \E_f$ lifts uniquely against each acyclic cofibration
  $\eta_X \colon X \rightarrow RX$, whence $\E(X,Y) \cong \E_f(RX,Y)$,
  naturally in $X$ and $Y$. Thus $\E_f$ is reflective in $\E$, and
  so complete and cocomplete since $\E$ is; we argue dually for
  $\E_c$. Finally, if $\E$ is combinatorial, then $\E_c$ and $\E_f$
  are locally presentable by Propositions~\ref{prop:20}
  and~\ref{prop:28}.
  \end{proof}
Moreover, we have the following homotopical properties:
\begin{Prop}
  \label{prop:5}
  The following are true in a one-dimensional model category:
  \begin{enumerate}[(i)]
  \item Every map between (co)fibrant objects is a (co)fibration.
  \item $Q$ and $R$ preserve and reflect weak equivalences.
  \item $R$ preserves cofibrations and inverts acyclic cofibrations.
  \item $Q$ preserves fibrations and inverts acyclic fibrations.
  \item Any weak equivalence between fibrant--cofibrant objects is an
    isomorphism.
  \item There is a natural isomorphism $QR \cong RQ$.
  \item A map is a weak equivalence if and only if it is inverted by
    $QR \cong RQ$.
  \end{enumerate}
\end{Prop}
\begin{proof}
  For (i), apply Proposition~\ref{prop:17}(iv) and (v) to composites
  $0 \rightarrow A \rightarrow B$ and $A \rightarrow B \rightarrow 1$.
  (ii) is standard in any model category. For (iii), consider the
  square
  \begin{equation}\label{eq:1}
    \cd{
      {A} \ar[r]^-{\eta_A} \ar[d]_{f} &
      {RA} \ar[d]^{Rf} \\
      {B} \ar[r]^-{\eta_B} &
      {RB}\rlap{ .}
    }
  \end{equation}
  Since $\eta_A$ and $\eta_B$ are (acyclic) cofibrations, if $f$ is a
  cofibration, then so is $Rf$ by Proposition~\ref{prop:17}(iv). If
  $f$ is moreover acyclic, then $Rf$ is both an acyclic cofibration
  and a fibration, whence invertible. Now (iv) is dual to (iii).
  For (v), note that any weak equivalence
  between cofibrant--fibrant objects is also a cofibration and a
  fibration, whence invertible. For (vi), note that by~(iii), $QRX$ is
  fibrant and $Q\eta_X \colon QX \rightarrow QRX$ is an acyclic
  cofibration; so by the uniqueness of the $(\T\C, \F)$-factorisation
  of $QX \rightarrow 1$, we must have $QRX \cong RQX$. Finally, (vii)
  follows from (ii) and (v) as $QR \cong RQ$ preserves and reflects
  weak equivalences.
\end{proof}

\section{Localities for one-dimensional model structures}
\label{sec:local-coloc-one}

We now begin to investigate the process of localisation for
one-dimensional combinatorial model categories. First we fix our
terminology.
\begin{Defn}
  \label{def:7}
  A \emph{combinatorial localisation} of a combinatorial
  one-dimensional model category $\E$ is a combinatorial
  one-dimensional model category $\E_\ell$ with the same underlying
  category, the same cofibrations, and at least as many acyclic
  cofibrations.
\end{Defn}

As in the introduction, a localisation $\E_\ell$ of $\E$
is completely determined by its subcategory $(\E_\ell)_f$ of fibrant
objects. Our objective in this section is to show that, in the
one-dimensional combinatorial setting, the subcategories so arising
are captured perfectly by the following notion of \emph{homotopical
  locality}.

\begin{Defn}
  \label{def:1}
  A \emph{locality} for a combinatorial one-dimensional model
  category $\E$ is a full subcategory $\E_{\ell f} \subseteq \E_f$,
  whose objects we call \emph{local}, such that:
  \begin{enumerate}[(i)]
  \item $\E_{\ell f}$ is locally presentable and reflective in $\E$
    via a reflector $\upsilon \colon 1 \rightarrow R_\ell$;
  \item If $X, Y \in \E_f$ are weakly equivalent, then $X$ is local
    just when $Y$ is.
  \end{enumerate}
  We call a locality \emph{homotopical}
  if, in addition:
  \begin{enumerate}[(i)]
    \addtocounter{enumi}{2}
  \item $R_\ell$  preserves weak
  equivalences.
  \end{enumerate}
\end{Defn}
Note that $\E_f$ is itself a locality---indeed, the maximal one---and
is homotopical by Proposition~\ref{prop:5}(ii).
\begin{Rk}
  \label{rk:8}
  The data for a homotopical locality resemble the input data for the
  Bousfield--Friedlander approach to
  localisation~\cite[Theorem~A.7]{Bousfield1978Homotopy}. Their
  setting also involves a right properness axiom which ensures that
  the necessary factorisations can be constructed in an elementary
  fashion; however, as noted in~\cite{Stanculescu2008Note}, this axiom
  can be dropped in the combinatorial setting, at the cost of losing
  an explicit grasp on the factorisations. The axioms for a
  homotopical locality above are a one-dimensional version of this
  more general form of the Bousfield--Friedlander axioms.
\end{Rk}
The easier direction is that any localisation gives rise to a
homotopical locality:

\begin{Prop}
  \label{prop:24}
  Let $\E$ be a combinatorial one-dimensional model category, and let
  $\E_\ell$ be a combinatorial localisation of $\E$. The
  subcategory $(\E_{\ell})_f$ of $\E_\ell$-fibrant objects is a
  homotopical locality for $\E$.
\end{Prop}
\begin{proof}
  $(\E_{\ell})_f \subseteq \E_f$ since $\E_\ell$ has more acyclic
  cofibrations than $\E$; it is of course full, replete and reflective
  in $\E$, and is locally presentable by Proposition~\ref{prop:22}. We
  next verify (iii). If $f$ is an $\E$-weak equivalence, then it is an
  $\E_\ell$-weak equivalence, whence by Proposition~\ref{prop:5}(vii)
  inverted by $QR_\ell$. Since $Q$ is also the cofibrant replacement
  for $\E$, we conclude by Proposition~\ref{prop:5}(ii) that
  $R_\ell(f)$ is an $\E$-weak equivalence. Finally, for (ii), let
  $f \colon X \rightarrow Y$ be a weak equivalence in $\E_f$ and
  consider the square:
  \begin{equation*}
    \cd[@C+0.5em@-0.5em]{
      {X} \ar[r]^-{f} \ar[d]_{\upsilon_{X}} &
      {Y} \ar[d]^{\upsilon_Y} \\
      {R_\ell X} \ar[r]^-{R_\ell f} &
      {R_\ell Y}\rlap{ .}
    }
  \end{equation*}
  We must show $\upsilon_X$ is invertible just when $\upsilon_Y$ is.
  But both are (acyclic) cofibrations in $\E_\ell$, whence
  cofibrations in $\E$, between $\E$-fibrant objects; so they are
  invertible just when they are $\E$-weak equivalences. But 
  $R_\ell$ preserves $\E$-weak equivalences, so both horizontal maps are
  $\E$-weak equivalences, whence $\upsilon_{X}$ is an $\E$-weak
  equivalence if and only if $\upsilon_Y$ is.
\end{proof}

To show conversely that every homotopical locality $\E_{\ell f}$
arises from a localisation, we carry out the procedure outlined in the
introduction: lifting the model structure from $\E$ to $\E_{\ell f}$
and back again, and mixing the result with the original model
structure. Before doing so, we establish some necessary
properties of localities.

\begin{Lemma}
  \label{lem:10}
  If $\E_{\ell f}$ is a locality for $\E$, then:
  \begin{enumerate}[(i)]
  \item $Q$ preserves and reflects locality of fibrant objects;
  \item Each $\upsilon_X \colon X \rightarrow R_\ell X$ is a cofibration;
  \item $R_\ell$ preserves cofibrations.
  \end{enumerate}
\end{Lemma}
\begin{proof}
  For (i), apply property (ii) of a locality to
  $\varepsilon_X \colon QX \rightarrow X$. For (ii), we factor
  $\upsilon_X$ as a cofibration $f \colon X \rightarrow P$ followed by
  an acyclic fibration $g \colon P \rightarrow R_\ell X$. Now $P$ is
  fibrant since $R_\ell X$ is, and $R_\ell X$ is local; so applying
  property (ii) of a locality to $g$, we conclude that $P$ is local.
  We can thus extend $f \colon X \rightarrow P$ to a map
  $h \colon R_\ell X \rightarrow P$ with $h\upsilon_X = f$; now
  $\upsilon_X = gf = gh\upsilon_X$, whence $gh = 1$, and so
  \begin{equation*}
    \cd{
      X \ar@{=}[r]^-{} \ar[d]_-{\upsilon_X} &
      X \ar@{=}[r]^-{} \ar[d]_-{f} &
      X \ar[d]_-{\upsilon_X} \\
      R_\ell X \ar[r]^-{h} & P \ar[r]^-{g} & R_\ell X
    }
  \end{equation*}
  exhibits $\upsilon_X$ as a retract of the cofibration $f$ and so 
  a cofibration. Finally, for (iii), apply (ii) and Proposition~\ref{prop:17}(iv) to the naturality square for
  $R_\ell$ at a cofibration. 
\end{proof}

We now show that the model structure on $\E$ lifts along the inclusion
$\E_{\ell f} \rightarrow \E$.

\begin{Prop}
  \label{prop:25}
  If $\E_{\ell f}$ is a locality for the one-dimensional
  combinatorial $\E$, then the model structure on $\E$ restricts to
  one on $\E_{\ell f}$, with classes as follows:
  \begin{itemize}
  \item Cofibrations = $\mathrm{LLP}$(maps inverted by $Q$);
  \item Acyclic cofibrations = isomorphisms;
  \item Fibrations = all maps;
  \item Acyclic fibrations = weak equivalences = maps inverted by
    $Q$.
  \end{itemize}
  The restricted model structure is one-dimensional and
  projectively lifts that on $\E$.
\end{Prop}
\begin{proof}
  On taking either factorisation of a map between objects of
  $\E_{\ell f}$, the interposing object clearly lies in $\E_f$, and is
  moreover local by Lemma~\ref{lem:10}(i). So the model structure
  restricts; in particular, it is a projective lifting along
  $\E_{\ell f} \rightarrow \E$ and so also one-dimensional. Since
  $\E_{\ell f} \subseteq \E_f$, Proposition~\ref{prop:5}(i) gives the
  characterisation of the restricted fibrations and acyclic
  cofibrations. The restricted weak equivalences therefore equal
  the restricted acyclic fibrations; and as every local object
  is fibrant, these are by Proposition~\ref{prop:5}(ii), (iii) and
  (v), exactly the maps inverted by $Q$. 
\end{proof}

And now we lift back in the other direction:

\begin{Prop}
  \label{prop:10}
  Let $\E_{\ell f}$ be a homotopical locality for $\E$. The
  restricted model structure on $\E_{\ell f}$ lifts injectively along
  $R_\ell \colon \E \rightarrow \E_{\ell f}$; this new model structure
  $\E'$ is one-dimensional, with acyclic cofibrations the maps
  inverted by $R_\ell$, and with fibrant objects the local objects.
  The identity functor $\E \rightarrow \E'$ is left Quillen.
\end{Prop}
\begin{proof}
  By assumption, $\E_{\ell f}$ is locally presentable and $R_\ell$ is
  a left adjoint. Now as $R_\ell$ preserves weak equivalences by
  assumption and cofibrations by Lemma~\ref{lem:10}(iii), we have
  $\W \subseteq R_\ell^{-1}(\W)$ and $\C \subseteq R_\ell^{-1}(\C)$, whence
  $\mathrm{RLP}(R_\ell^{-1}(\C)) \subseteq \mathrm{RLP}(\C) = \T\F
  \subseteq \W \subseteq R_\ell^{-1}(\W)$. Thus by
  Proposition~\ref{prop:1}, the injectively lifted model structure
  $\E'$ exists,
  and is one-dimensional by Proposition~\ref{prop:17}(iv); since $\E'$
  has more cofibrations and acyclic cofibrations than $\E$, the
  identity $\E \rightarrow \E'$ is left Quillen. The characterisation
  of the acyclic cofibrations follows since the $\E_{\ell f}$-acyclic
  cofibrations are the isomorphisms. Finally, an object is
  $\E'$-fibrant if and only if it is orthogonal to all maps inverted
  by the reflector $R_\ell \colon \E \rightarrow \E_{\ell f}$, if and
  only if it is in $\E_{\ell f}$.
\end{proof}

Finally, we mix this new model structure with our original one:

\begin{Prop}
  \label{prop:12}
  If $\E_{\ell f}$ is a homotopical locality for $\E$, then there
  exists a combinatorial localisation $\E_\ell$ of $\E$ such that
  $(\E_\ell)_f = \E_{\ell f}$.
\end{Prop}
\begin{proof}
  The model structure $\E'$ of the last result has more cofibrations
  and weak equivalences than $\E$; so we can mix with $\E$ to obtain a
  model structure $\E_\ell$ with the cofibrations of $\E$ and the weak
  equivalences of $\E'$. It remains to show
  $(\E_\ell)_f = \E_{\ell f}$. As the $\E'$-fibrant objects are the
  local ones, and as $\T\C_\ell \subseteq \T\C'$, each local object is
  $\E_\ell$-fibrant. Conversely, if $X$ is $\E_\ell$-fibrant, then
  since $\upsilon_X \colon X \rightarrow R_\ell X$ is in
  $\T\C_\ell$---being both an $\E$-cofibration and inverted by
  $R_\ell$---the identity $X \rightarrow X$ extends to a retraction
  $p \colon RX \rightarrow X$ for $\upsilon_X$. Now since
  $\upsilon_X p \upsilon_X = \upsilon_X$ also $\upsilon_X p = 1$ and
  so $X$ is local as an isomorph in $\E_f$ of the local $RX$.
\end{proof}

Putting together the preceding results, we obtain:

\begin{Thm}
  \label{thm:1}
  Let $\E$ be a one-dimensional combinatorial model category. The
  assignation $\E_\ell \mapsto (\E_\ell)_f$ yields an order-reversing
  bijection between combinatorial localisations of $\E$ (ordered by
  inclusion of their acyclic cofibrations) and homotopical localities
  for $\E$ (ordered by inclusion of their subcategories of local
  objects).
\end{Thm}
\begin{proof}
  The assignation is well-defined by Proposition~\ref{prop:24},
  clearly order-reversing, and surjective by
  Proposition~\ref{prop:12}. Moreover, two localisations of $\E$ which
  induce the same localities must have isomorphic cofibrant--fibrant
  replacement functors, whence by Proposition~\ref{prop:5}(vii) the same
  weak equivalences; and so must coincide.
\end{proof}

\section{Left  properness}
\label{sec:left-right-prop}

Localisation of model structures is often carried out under the
assumption of left properness; recall that a model structure is called
\emph{left proper} if the pushout of a weak equivalence along a
cofibration is a weak equivalence. We now explain the significance of
this condition in the one-dimensional context, by proving:

\begin{Prop}
  \label{prop:26}
  If $\E$ is one-dimensional, combinatorial and left proper, then any
  locality for $\E$ is homotopical.
\end{Prop}

Before proving this, we establish some preparatory lemmas. In stating
the first, note that, as a special case of Proposition~\ref{prop:25},
any one-dimensional model structure on a category $\E$ restricts to a
one-dimensional model structure on $\E_f$.

\begin{Lemma}
  \label{lem:8}
  A one-dimensional model category $\E$ is left proper if and only if
  the restricted model structure on $\E_f$ is left proper.
\end{Lemma}
\begin{proof}
  $R \colon \E \rightarrow \E_f$ preserves pushouts, weak equivalences
  and cofibrations; so if $\E_f$ is left proper, then the cobase
  change in $\E$ of a weak equivalence along a cofibration is sent by
  $R$ to another such cobase change in $\E_f$ and so, by left
  properness, to a weak equivalence. As $R$ reflects weak
  equivalences, this shows $\E$ is left proper.

  Conversely, if $\E$ is left proper, then the cobase change in $\E_f$
  of a weak equivalence along a cofibration may be calculated by
  forming the cobase change in $\E$---which yields a weak equivalence by
  left properness---and then applying $R$---which yields a weak
  equivalence since $R$ preserves such. This shows $\E_f$ is left proper.
\end{proof}
\begin{Rk}
  \label{rk:5}
  It follows that \emph{any localisation of the left proper
    one-dimensional $\E$ is left proper}. For indeed, by copying the
  ``only if'' direction of the preceding proof, we see that the
  restriction of the model structure $\E$ to $(\E_{\ell})_f$ is left
  proper. But this model structure is equally the restriction of the
  model structure $\E_\ell$ to $(\E_\ell)_f$ and so by the ``if''
  direction of the preceding result, $\E_\ell$ is also left proper.
\end{Rk}

\begin{Lemma}
  \label{lem:7}
  If $\E_{\ell f}$ is a locality for the left proper $\E$, then
  for any weak equivalence $f \colon X \rightarrow Y$ between fibrant
  objects, the following is a pushout in $\E_f$:
  \begin{equation}\label{eq:7}
    \cd{
      {X} \ar[r]^-{f} \ar[d]_{\upsilon_X} &
      {Y} \ar[d]^{\upsilon_Y} \\
      {R_\ell X} \ar[r]^-{R_\ell f } &
      {R_\ell Y}
    }
  \end{equation}
\end{Lemma}
\begin{proof}
  Let us form the pushout
  \begin{equation*}
    \cd{
      {X} \ar[r]^-{f} \ar[d]_{\upsilon_X} &
      {Y} \ar[d]^{p} \\
      {R_\ell X} \ar[r]^-{q} &
      {P} \pushoutcorner
    }
  \end{equation*}
  in the left proper $\E_f$. Since $f$ is a weak equivalence and
  $\upsilon_X$ is a cofibration, $q$ is also a weak equivalence in
  $\E_f$; thus since $R_\ell X$ is local, $P$ is too. Moreover, since
  $\upsilon_X$ has the left lifting property against any local object,
  so does its pushout $p$; whence $p \colon Y \rightarrow P$ is a
  reflection of $Y$ into $\E_{\ell f}$. As
  $\upsilon_Y \colon Y \rightarrow R_\ell Y$ is another such
  reflection, the unique induced map $P \rightarrow R_\ell Y$ is thus
  invertible.
\end{proof}

We can now give:
\begin{proof}[Proof of Proposition~\ref{prop:26}]
  Let $\E_{\ell f}$ be a locality. If $g \colon X \rightarrow Y$ is a
  weak equivalence in $\E$, then $Rg$ is one in $\E_f$, and so taking
  $f = Rg$ in~\eqref{eq:7} shows that $R_\ell Rg$, as a pushout of a
  weak equivalence along a cofibration in the left proper $\E_f$, is
  also a weak equivalence. Now $\eta_X \colon X \rightarrow RX$ and
  $\eta_Y \colon X \rightarrow RY$ have the unique left lifting
  property against every (fibrant and so every) local object, and as
  such are inverted by $R_\ell$; whence 
  \begin{equation*}
    R_\ell g = R_\ell X \xrightarrow{R_\ell \eta_X} R_\ell RX
    \xrightarrow{R_\ell Rg} R_\ell RY \xrightarrow{(R_\ell \eta_Y)^{-1}}
    R_\ell Y
  \end{equation*}
  is a composite of weak equivalences and so a weak equivalence.
\end{proof}

We have thus shown that, in the left proper context, we can drop the
modifier ``homotopical'' from the statement of Theorem~\ref{thm:1}:
that is, localisations of the left proper $\E$ correspond to
localities on $\E$. The value of this is that localities are rather
easy to construct, by virtue of:

\begin{Prop}
  \label{prop:14}
  Let $\E$ be a combinatorial one-dimensional model category. The
  assignation $\E_{\ell f} \mapsto \E_{\ell f} \cap \E_c$ yields an
  order-preserving bijection between localities for $\E$ and full,
  replete, reflective, locally presentable subcategories of $\E_{cf}$
  (where in each case the order is given by inclusion of
  subcategories).
\end{Prop}
\begin{proof}
  If $\E_{\ell f}$ is a locality, then by Lemma~\ref{lem:10}(iii) its
  reflector $R_\ell$ maps $\E_{cf}$ into $\E_{\ell f} \cap \E_c$, so
  that $\E_{\ell f} \cap \E_c$ is reflective in $\E_{cf}$ via
  $R_\ell$. Moreover, since $\E_{\ell f}$ is reflective in $\E_f$ via
  $R_\ell$, the functor $R_\ell \colon \E_f \rightarrow \E_f$ is
  accessible by Proposition~\ref{prop:28}. Since $\E_{cf}$ is
  coreflective in $\E_f$ and hence closed under colimits,
  $R_\ell \colon \E_{cf} \rightarrow \E_{cf}$ is also accessible and
  hence $\E_{\ell f} \cap \E_c$ is locally presentable by
  Proposition~\ref{prop:28} again.

  This shows that $\E_{\ell f} \mapsto \E_{\ell f} \cap \E_c$ is
  well-defined, and it is injective since for any locality, the
  objects in $\E_{\ell f}$ are, by Lemma~\ref{lem:10}(i), those
  $X \in \E_f$ for which $QX \in \E_{\ell f} \cap \E_c$. To show
  surjectivity, let $\E_{\ell cf} \subseteq \E_{cf}$ be reflective and
  locally presentable. Form the pullback
  \begin{equation*}
    \cd{
      {\E_{\ell f}} \pullbackcorner  \ar[r]_-{i'} \ar[d]_{Q'} &
      {\E_f} \ar[d]^{Q} \\
      {\E_{\ell cf}} \ar[r]^-{i} &
      {\E_{cf} \rlap{ ;}}
    }
  \end{equation*}
  so objects of $\E_{\ell f}$ are objects $X \in \E_f$ with
  $QX \in \E_{\ell cf}$. Since $QX \cong X$ whenever $X \in \E_c$, we
  have $\E_{\ell f} \cap \E_c = \E_{\ell cf}$, and so will be done so
  long as $\E_{\ell f}$ is a locality.

  To check condition (i) of Definition~\ref{def:1}, note that
  $\E_{\ell cf}$, $\E_{cf}$ and $\E_f$ are locally presentable, $i$
  and $Q$ are right adjoints, and $i$ is an isofibration; so
  by~\cite[Theorem~2.18]{Bird1984Limits}, $\E_{\ell f}$ is also
  locally presentable, and $i'$ is also a right adjoint. So
  $\E_{\ell f}$ is reflective in $\E_f$, and $\E_f$ is reflective in
  $\E$, whence $\E_{\ell f}$ is reflective in $\E$. To check
  Definition~\ref{def:1}(ii), note that if $f \colon X \rightarrow Y$
  is a weak equivalence in $\E_f$, then $Qf \colon QX \rightarrow QY$
  is invertible in $\E_{cf}$; since $\E_{\ell cf}$ is replete in
  $\E_{cf}$, we thus have that $X \in \E_{\ell f}$ iff
  $QX \in \E_{\ell cf}$ iff $QY \in \E_{\ell cf}$ iff
  $Y \in \E_{\ell f}$.
\end{proof}

Combining this with Theorem~\ref{thm:1} and Proposition~\ref{prop:26},
we therefore obtain: \firstmaintheorem

\section{Localisation at a set of maps}
\label{sec:relat-enrich-bousf}

In practice, one often constructs localisations of a left proper model
category starting from a set of maps which one wishes to make into
weak equivalences. We now use the theory of the preceding section to
reproduce this construction in the one-dimensional context. First we
recall the basic definitions:
\begin{Defn}
  \label{def:3}
  If $\E$ is a model category enriched over the monoidal model
  category $\V$, then the \emph{derived hom} of $\E$ is the functor
  \begin{equation*}
    \E_h \colon \E^\mathrm{op} \times \E \xrightarrow{\E(Q\thg, R\thg)} \V
    \xrightarrow{\mathrm{Ho}} \mathrm{Ho}\ \V\rlap{ .}
  \end{equation*}
\end{Defn}

For enrichment over the model structure for $0$-types on
$\cat{Set}$, the functor
$\mathrm{Ho} \colon \cat{Set} \rightarrow \mathrm{Ho}\ \cat{Set}$ is
the identity, so that for a one-dimensional model category $\E$ we
have $\E_h(A,B) = \E(QA,RB)$. 

\begin{Defn}
  \label{def:4}
  If $\E$ is a model $\V$-category, $X \in \E$ and $f \in \E(A,B)$,
  then we write $f \ho X$ if
  $\E_h(f, X) \colon \E_h(B, X) \rightarrow \E_h(A,X)$ is invertible.
  Given a class of maps $S$ in $\E$, we now say that:
  \begin{itemize}
  \item An object $X \in \E$ is \emph{$S$-local} if 
    $f \ho X$ for all $f \in S$;
  \item A map $f \in \E$ is an \emph{$S$-local equivalence} if $f \ho
    X$ for all $S$-local $X \in \E$.
  \end{itemize}
\end{Defn}
\begin{Rk}
  \label{rk:4}
  The derived hom $\E_h$ has the property of sending weak equivalences
  in each variable to isomorphisms; in particular, we have
  \begin{equation*}
    \E_h(A,B) \cong \E_h(A,QB) \cong \E_h(A,RB) \ \ \text{and} \ \ 
    \E_h(A,B) \cong \E_h(RA,B) \cong \E_h(QA,B)\rlap{ .}
  \end{equation*}
  It follows that $Q$ and $R$ preserve and reflect both $S$-local
  objects and $S$-local equivalences. As a consequence, in showing
  that a map $f$ is an $S$-local equivalence, it suffices to check
  that $f \ho X$ for each \emph{cofibrant--fibrant} $S$-local $X$.
\end{Rk}

\begin{Rk}
  \label{rk:3}
  If $\E_\ell$ is any combinatorial localisation of $\E$ whatsoever,
  then each $\E$-weak equivalence is an $\E_\ell$-weak equivalence. In
  particular, both $\eta \colon 1 \rightarrow R$ and
  $\varepsilon \colon Q \rightarrow 1$ are pointwise $\E_\ell$-weak
  equivalences so that, by the two-out-of-three property for
  $\E_\ell$-weak equivalences, \emph{$Q$ and $R$ preserve and reflect
    $\E_\ell$-equivalences}.
\end{Rk}

\begin{Thm}
  \label{thm:5}
  Let $\E$ be combinatorial, left proper and one-dimensional. For any
  set of maps $S$ of $\E$, there exists a combinatorial
  localisation $\E_\ell$ of the model structure $\E$
  for which:
  \begin{itemize}
  \item The fibrant objects are the $S$-local $\E$-fibrant objects;
  \item The weak equivalences are the $S$-local equivalences.
  \end{itemize}
  Moreover, every combinatorial localisation of $\E$
  arises thus.
\end{Thm}
\begin{proof}
  Given a set $S$ of maps, let $\E_{\ell cf} \subseteq \E_{cf}$ be the
  full subcategory of $S$-local fibrant--cofibrant objects. Note that
  $X \in \E_{cf}$ is in $\E_{\ell cf}$ just when $\E_{cf}(QRf, X)$ is
  invertible for each $f \in S$. Thus, on taking
  $S' = \{QRf : f \in S\}$ and
  \begin{equation*}
    I = J = S' \cup 
    \{\nabla_{g} \colon B +_{A} B \rightarrow B \mid g \colon A
    \rightarrow B \in S'\}
  \end{equation*}
  as generating (acyclic) cofibrations, the small object argument
  yields a combinatorial one-dimensional model structure on $\E_{cf}$
  with subcategory of fibrant objects $\E_{\ell cf}$. So by
  Proposition~\ref{prop:22}, $\E_{\ell cf}$ is reflective in $\E_{cf}$
  and locally presentable.

  Now, applying Theorem~\ref{thm:3} to $\E_{\ell cf}$ yields a
  localisation $\E_\ell$ of $\E$ with $(\E_\ell)_{cf} = \E_{\ell cf}$.
  As $Q$ preserves and reflects both the $\E_\ell$-fibrancy and
  the $S$-locality of $\E$-fibrant objects, the $\E_\ell$-fibrant objects
  are the $S$-local $\E$-fibrant ones. Moreover, the $\E_\ell$-weak
  equivalences in $\E_{cf}$ are the maps inverted by the reflector
  into $\E_{\ell cf}$, which are those $f$ such that
  $\E_{cf}(f, X) \cong \E_h(f,X)$ is invertible for all
  $X \in \E_{\ell cf}$---which, by Remark~\ref{rk:4}, are exactly the
  $S$-local equivalences in $\E_{cf}$. Since, by Remark~\ref{rk:4} and
  Remark~\ref{rk:3}, $QR$ preserves and reflects both $S$-local
  equivalences and $\E_\ell$-weak equivalences, it follows that the
  $\E_\ell$-weak equivalences are the $S$-local equivalences.
  
  Finally, if $\E_\ell$ is any localisation of $\E$, then by
  Theorem~\ref{thm:3}, $(\E_\ell)_{cf}$ is locally presentable and
  reflective in $\E_{cf}$. Thus,
  by~\cite[Theorem~1.39]{Adamek1994Locally}, there is a set $S$ of
  maps in $\E_{cf}$ so that $(\E_\ell)_{cf}$ comprises those
  $X \in \E_{cf}$ for which $\E_{cf}(\thg, X) \cong \E_h(\thg, X)$
  inverts each $g \in S$---in other words, the $S$-local
  cofibrant--fibrant objects. As a model structure is determined by
  its cofibrations and cofibrant--fibrant objects, $\E_\ell$ is thus
  the localisation of $\E$ at~$S$.
\end{proof}

\section{Colocalisation}
\label{sec:colocalisation}
As noted in the introduction, an advantage of our approach is that
everything we have done adapts without fuss from the case of
localisations to \emph{colocalisations}.

\begin{Defn}
\label{def:10}
A \emph{combinatorial colocalisation} of a combinatorial
one-dimensional model category $\E$ is a combinatorial one-dimensional
model category $\E_{r}$ with the same underlying category, the same
fibrations, and at least as many acyclic fibrations.
\end{Defn}

The arguments of Section~\ref{sec:local-coloc-one} dualise immediately
to show that colocalisations correspond to homotopical
\emph{colocalities}:

\begin{Defn}
  \label{def:2}
  A \emph{colocality} for a combinatorial one-dimensional model
  category $\E$ is a full subcategory $\E_{r c} \subseteq \E_c$, whose
  objects we call \emph{colocal}, such that:
  \begin{enumerate}[(i)]
  \item $\E_{r c}$ is locally presentable and coreflective in $\E$
    via a coreflector $\xi \colon Q_r \rightarrow 1$;
  \item If $X, Y \in \E_c$ are weakly equivalent, then $X$ is colocal
    just when $Y$ is.
  \end{enumerate}
  We call a colocality \emph{homotopical}
  if, in addition:
  \begin{enumerate}[(i)]
    \addtocounter{enumi}{2}
  \item $Q_r$  preserves weak
  equivalences.
  \end{enumerate}
\end{Defn}
\begin{Thm}
  \label{thm:2}
  Let $\E$ be a one-dimensional combinatorial model category. The
  assignation $\E_r \mapsto (\E_r)_c$ yields an order-reversing
  bijection between combinatorial colocalisations of $\E$ (ordered by
  inclusion of their fibrations) and homotopical colocalities for $\E$
  (ordered by inclusion of their subcategories of colocal objects).
\end{Thm}

Now the arguments of Section~\ref{sec:left-right-prop} dualise to show
that \emph{every colocality for the right proper one-dimensional
  combinatorial $\E$ is homotopical}. The analogue of
Proposition~\ref{prop:14}, however, requires a proof which is not
exactly dual, and which we therefore give in more detail:
\begin{Prop}
  \label{prop:27}
  Let $\E$ be a combinatorial one-dimensional model category. The
  assignation $\E_{r c} \mapsto \E_{r c} \cap \E_f$ yields an
  order-preserving bijection between colocalities for $\E$ and full,
  replete, coreflective, locally presentable subcategories of
  $\E_{cf}$ (where in each case the order is given by inclusion of
  subcategories).
\end{Prop}
\begin{proof}
  The coreflectivity of $\E_{r c} \cap \E_f$ in $\E_{cf}$ is dual
  to before. For its local presentability, as $\E_{r c}$ is
  coreflective in $\E_c$ via $Q_r$, Proposition~\ref{prop:28}
  implies that $Q_r \colon \E_c \rightarrow \E_c$ preserves
  $\lambda$-filtered colimits for some $\lambda$; and as $\E_{cf}$ is
  reflective in $\E_c$, it is by Proposition~\ref{prop:28} closed in
  $\E_c$ under $\kappa$-filtered colimits for some
  $\kappa \geqslant \lambda$. So
  $Q_r \colon \E_{cf} \rightarrow \E_{cf}$ preserves
  $\kappa$-filtered colimits, and so $\E_{r c} \cap \E_f$ is
  locally presentable by Proposition~\ref{prop:28}. Thus
  $\E_{r c} \mapsto \E_{r c} \cap \E_f$ is well-defined, and it
  is injective as before; for surjectivity, given
  $\E_{r cf} \subseteq \E_{cf}$ coreflective and locally
  presentable, we form the pullback
  \begin{equation*}
    \cd{
      {\E_{r c}} \pullbackcorner  \ar[r]_-{j'} \ar[d]_{R'} &
      {\E_c} \ar[d]^{R} \\
      {\E_{r cf}} \ar[r]^-{j} &
      {\E_{cf} \rlap{ ;}}
    }
  \end{equation*}
  now the previous argument will carry over, \emph{mutatis mutandis},
  so long as $\E_{r c}$ is locally presentable and $j'$ has a right
  adjoint. Since $R$ and $j$ are left adjoint functors between locally
  presentable categories, this follows like before but now appealing
  to Theorem~3.15, rather than Theorem~2.18, of~\cite{Bird1984Limits}.
\end{proof}

Putting these results together, we now obtain:

\secondmaintheorem

Analogously to Section~\ref{sec:relat-enrich-bousf}, one often
constructs colocalisations of a right proper model category from a set
of objects which generate the colocal ones under homotopy colimits. We
now rederive this result in the one-dimensional setting.

\begin{Defn}
  \label{def:8}
  Given a model $\V$-category $\E$, an object $X \in \E$ and a map
  $f \in \E(A,B)$, we write $X \ho f$ if
  $\E_h(X, f) \colon \E_h(X, A) \rightarrow \E_h(X,B)$ is invertible.
  Given a class of objects $K$ in $\E$, we now say that:
  \begin{itemize}
  \item A map $f \in \E$ is a \emph{$K$-colocal equivalence} if $X \ho
    f$ for all $X \in K$;
  \item An object $X \in \E$ is \emph{$K$-colocal} if 
    $X \ho f$ for all $K$-colocal equivalences $f \in \E$.
  \end{itemize}
\end{Defn}

\begin{Thm}
  \label{thm:6}
  Let $\E$ be combinatorial, right proper and one-dimensional. For any
  set of objects $K$ in $\E$, there exists a combinatorial
  colocalisation $\E_{r}$ of the model structure
  $\E$ for which:
  \begin{itemize}
  \item The cofibrant objects are the $K$-colocal $\E$-cofibrant objects;
  \item The weak equivalences are the $K$-colocal equivalences.
  \end{itemize}
  Moreover, every combinatorial colocalisation of $\E$
  arises thus.
\end{Thm}
\begin{proof}
  Given a set $K$ of objects, let $\E_{rcf} \subseteq \E_{cf}$ be the
  full subcategory of $K$-colocal fibrant--cofibrant objects.
  Taking $K' = \{QRX : X \in
  K\}$ and taking
  \begin{equation}\label{eq:2}
    I = J = \{ 0 \rightarrow Y : Y \in K'\} \cup 
    \{\nabla \colon Y + Y \rightarrow Y \mid Y \in K'\}\rlap{ ,}
  \end{equation}
  we obtain by the small object argument a combinatorial
  one-dimensional model structure on $\E_{cf}$ with acyclic fibrations
  the $K$-colocal equivalences in $\E_{cf}$, and so, by the dual of
  Remark~\ref{rk:4}, with cofibrant objects the $K$-colocal objects in
  $\E_{cf}$. Thus, by Proposition~\ref{prop:22}, $\E_{rcf}$ is
  coreflective in $\E_{cf}$ and locally presentable, and so applying
  Theorem~\ref{thm:4} to $\E_{rcf}$ yields a colocalisation $\E_r$ of
  $\E$ with $(\E_r)_{cf} = \E_{r cf}$.

  The same argument as previously shows that the $\E_r$-cofibrant
  objects are the $K$-colocal $\E$-cofibrant ones. Moreover, the
  $\E_r$-weak equivalences in $\E_{cf}$ are the maps inverted by the
  coreflector into $\E_{r cf}$, which are those $f$ such that
  $\E_{cf}(X, f) \cong \E_h(X,f)$ is invertible for all
  $X \in \E_{rcf}$. By the dual of Remark~\ref{rk:4}, these are
  exactly the $K$-colocal equivalences in $\E_{cf}$; so arguing as
  before, the $\E_r$-weak equivalences are the $K$-colocal
  equivalences.
  
  Finally, if $\E_r$ is any colocalisation of $\E$, then by
  Theorem~\ref{thm:4}, $(\E_r)_{cf}$ is locally presentable and
  coreflective in $\E_{cf}$. Since $(\E_r)_{cf}$ is locally
  presentable, it has a small full subcategory $\A$ whose
  colimit-closure in $(\E_r)_{cf}$ is the whole category; thus, since
  $(\E_r)_{cf}$ is closed in $\E_{cf}$ under colimits, the
  colimit-closure of $\A$ in $\E_{cf}$ is $(\E_r)_{cf}$. Now let
  $K = \ob \A$. The $K$-colocal objects in $\E_{cf}$ comprise a
  coreflective subcategory, which is colimit-closed, and so includes
  every object in $(\E_r)_{cf}$. On the other hand, each $K$-colocal
  object is a retract of an $I$-cell complex with $I$ as
  in~\eqref{eq:2}, so constructible from objects in $\A$ via colimits,
  and so in $(\E_r)_{cf}$. So the subcategory of $K$-colocal objects
  in $\E_{cf}$ is precisely $(\E_r)_{cf}$. As a model structure is
  determined by its fibrations and cofibrant--fibrant objects, $\E_r$
  is thus the colocalisation of $\E$ with respect to~$K$.
\end{proof}

\section{Examples}
\label{sec:examples}

We conclude this paper by describing some examples of one-dimensional
model categories obtained via Bousfield (co)localisation. While the
one-dimensionality means that there is no real homotopy theory,
we can at least find examples in which the fibrant, cofibrant or
fibrant--cofibrant objects are mathematically interesting.

As a first step, we may apply Theorem~\ref{thm:3} to see that
combinatorial localisations of the \emph{discrete} model structure on
a locally presentable category $\E$ correspond bijectively with full,
replete, reflective, locally presentable subcategories of $\E$; this
recovers Theorem~4.3 of~\cite{Salch2017The-Bousfield}\footnote{Or
  rather, its restriction to the combinatorial case; when starting
  from a discrete model structure, it is possible to construct
  (co)localisations under rather weaker assumptions than
  combinatoriality.}. The localised model structure corresponding to
the subcategory $\B$ is obtained by lifting the discrete model
structure on $\B$ injectively along the reflector
$R \colon \E \rightarrow \B$. This model structure is always left
proper, since every object is cofibrant, but with an eye towards
subsequent \emph{co}localisation, it will be useful to know when it is
also \emph{right} proper.

\begin{Defn}
  \label{def:9}
  A reflection $V \colon \B \leftrightarrows \E \colon F$ is called
  \emph{semi-left-exact} if the reflector $F \colon \E \rightarrow \B$
  preserves pullbacks along maps in the essential image of $V$.
\end{Defn}
This definition originates in Section 4
of~\cite{Cassidy1985Reflective}; the following result, describing the
relation with right proper model structures, was first observed
in~\cite{Rosicky2007Factorization}.

\begin{Lemma}
  \label{lem:1}
  A localisation of the discrete model structure on the locally
  presentable $\E$ is right proper if and only if the reflection
  $i \colon \E_{\ell f} \leftrightarrows \E \colon R_\ell$ is
  semi-left-exact.
\end{Lemma}

\begin{proof}
  The acyclic fibrations of the localised model structure are the
  isomorphisms, whence the weak equivalences are the acyclic
  cofibrations; so right properness is the condition that $\T\C$-maps
  are stable under pullback along $\F$-maps. Since the acyclic
  cofibrations are equally the maps inverted by $R_\ell$, its
  $(\T\C, \F)$-factorisation system is, in the terminology
  of~\cite{Cassidy1985Reflective}, the \emph{reflective factorisation
    system} corresponding to the subcategory $\E_{\ell f}$; now Theorem~4.3 of
  \emph{ibid}.~proves that $\T\C$-maps are stable under pullback along
  $\F$-maps just when the reflection is semi-left-exact.
\end{proof}

Putting this together with Theorem~\ref{thm:4}, we get:

\begin{Prop}
  \label{prop:15}
  Let $\A$, $\B$ and $\E$ be locally presentable. For any
  semi-left-exact reflection
  $i \colon \B \leftrightarrows \E \colon R$ and coreflection
  $j \colon \A \leftrightarrows \B \colon Q$ there is a one-dimensional
  model structure on $\E$ with fibrant objects those in the
  essential image of $i$, with cofibrant objects those
  $X \in \E$ such that $RX$ is in the essential image of $j$, and with
  cofibrant--fibrant objects those in the essential image of $ij$.
\end{Prop}

Dually, we can construct a model structure on $\E$ from a
semi-\emph{right}-exact \emph{co}reflection
$j \colon \B \leftrightarrows \E \colon Q$ together with a reflection
$i \colon \A \leftrightarrows \B \colon R$ by first colocalising and
then localising.

With these results in hand, we are now ready to give some examples. It
is readily checked that all of the categories we deal with are locally
presentable, and so we will make no mention of this in what follows.

\begin{Ex}
  \label{ex:2}
  Let $A$ be a commutative ring, and let $\cat{Zar}(A)$ denote the big
  Zariski topos of $A$. That is, $\cat{Zar}(A)$ the category of
  sheaves on the dual of the category
  $\smash{\cat{Alg}_{A}^\mathrm{fp}}$ of finitely presentable
  $A$-algebras, with the topology defined by surjective families of
  Zariski open inclusions. Sheafification gives a
  (semi-)left-exact reflection
  \begin{equation}\label{eq:10}
    \cd{
      {\cat{Zar}(A)} \ar@{ (->}@<-4.5pt>[r]_-{} \ar@{}[r]|-{\bot} &
      {[\smash{\cat{Alg}_{A}^\mathrm{fp}}, \cat{Set}]}\rlap{ .} \ar@<-4.5pt>[l]_-{}
    }
  \end{equation}

  Now let $\cat{zar}(A)$ denote the small Zariski topos of $A$: the
  category of sheaves on the dual of the subcategory
  $\cat{Loc}_A\subseteq \smash{\cat{Alg}_{A}^\mathrm{fp}}$ on the
  basic Zariski opens of $A$ (i.e., the localisations of $A$ at a
  single element) under the restricted Zariski topology. The inclusion
  $j \colon \cat{Loc}_A \rightarrow \smash{\cat{Alg}_{A}^\mathrm{fp}}$
  is fully faithful, left exact, and preserves and reflects covers; so
  by~\cite[Example~C2.3.23]{Johnstone2002Sketches2} there is a
  coreflection
  \begin{equation}
    \label{eq:11}
    \cd{
      {\cat{zar}(A)} \ar@{ (->}@<-4.5pt>[r]_-{} \ar@{}[r]|-{\top} &
      {\cat{Zar}(A)} \ar@<-4.5pt>[l]_-{}
    }
  \end{equation}
  with right adjoint given by restriction along $j$ and left adjoint
  by left Kan extension followed by sheafification. More concretely,
  the left adjoint sends $X \in \cat{zar}(A)$ to the functor of points
  of the $A$-scheme $p \colon \Lambda_X \rightarrow \mathrm{Spec}\,A$
  obtained by glueing Zariski open subschemes of $\mathrm{Spec}\,A$ in
  the manner specified by $X$. As such, we can see the image of
  this left adjoint as comprising the ``local homeomorphisms'' over
  $\mathrm{Spec}\,A$.

  Applying Proposition~\ref{prop:15} to~\eqref{eq:10}
  and~\eqref{eq:11}, we thus have a model structure on
  $[\smash{\cat{Alg}_{A}^\mathrm{fp}}, \cat{Set}]$ with fibrant
  objects the big Zariski sheaves and with as cofibrant--fibrant
  objects, the small Zariski sheaves seen as local homeomorphisms over
  $\mathrm{Spec}\, A$. The general fibrant objects are those functors
  of points $\smash{\cat{Alg}_{A}^\mathrm{fp}} \rightarrow \cat{Set}$
  whose sheafification lands in $\cat{zar}(A) \subseteq \cat{Zar}(A)$.
\end{Ex}

\begin{Ex}
  \label{ex:7}
  Let $k$ be an algebraically closed field, and let
  $\cat{LocArt}_k \subseteq \cat{Alg}_{A}^\mathrm{fp}$ denote the full
  subcategory on the local Artinian $k$-algebras. The topology on the
  dual of $\cat{LocArt}_k$ induced from the Zariski topology is easily
  seen to be \emph{discrete}, so that the category of sheaves thereon
  is equally the category of presheaves; now, as in the preceding
  example, we induce a coreflection
  \begin{equation}
    \label{eq:14}
    \cd{
      {[\cat{LocArt}_k, \cat{Set}]} \ar@{ (->}@<-4.5pt>[r]_-{} \ar@{}[r]|-{\top} &
      {\cat{Zar}(k)} \ar@<-4.5pt>[l]_-{}
    }
  \end{equation}
  whose right adjoint has a further right adjoint given by right Kan
  extension along the inclusion
  $\cat{LocArt}_k \subseteq \cat{Alg}_{A}^\mathrm{fp}$. It follows
  that this coreflection is semi-right-exact.

  The linear duals of local Artinian $k$-algebras are the
  cocommutative $k$-coalgebras which are finite-dimensional and
  \emph{irreducible}: that is, contain a unique grouplike element.
  By~\cite[Corollary~8.0.7]{Sweedler1969Hopf}, any cocommutative
  $k$-coalgebra is the direct sum of irreducible ones, and
  by~\cite[Theorem~2.2.1]{Sweedler1969Hopf}, any irreducible
  $k$-coalgebra is the union of its (irreducible) finite-dimensional
  subcoalgebras. It follows that the linear duals of local
  Artinian $k$-algebras are dense in the (cocomplete) category
  $k\text-\cat{Cocomm}$ of cocommutative coalgebras, and so we have a
  reflection
  \begin{equation}
    \label{eq:15}
    \cd{
      {k\text-\cat{Cocomm}} \ar@{ (->}@<-4.5pt>[r]_-{} \ar@{}[r]|-{\bot} &
      {[\cat{LocArt}_k, \cat{Set}]\rlap{ .}} \ar@<-4.5pt>[l]_-{}
    }
  \end{equation}
  Applying the dual of Proposition~\ref{prop:15} to~\eqref{eq:14}
  and~\eqref{eq:15}, we thus have a model structure on the big Zariski
  topos of $k$ whose cofibrant objects are the colimits in
  $\cat{Zar}(k)$ of the spectra of local Artinian $k$-algebras, and
  whose cofibrant--fibrant objects are cocommutative $k$-algebras; the
  inclusion into $\cat{Zar}(k)$ identifies these with the filtered
  colimits of the spectra of Artinian $k$-algebras. The general
  fibrant objects are Zariski sheaves $X$ satisfying a form of
  ``infinitesimal linearity''~\cite{Kock1981Synthetic} which is
  satisfied, for example, by any scheme over $\cat{Spec}(k)$. Among
  other things, this infinitesimal linearity ensures the set of
  tangent vectors $T_e(X)$ to a $k$-valued point
  $e \colon \cat{Spec}(k) \rightarrow X$---that is, the set of
  extensions of $e$ through the map
  $\cat{Spec}(k) \rightarrow
  \cat{Spec}(k[\varepsilon]/\varepsilon^2)$---has the structure of a
  $k$-vector space, which is moreover a Lie algebra if $e$ is the
  neutral element for a group structure on $X$.
\end{Ex}
\begin{Ex}
  \label{ex:3}
  Let $X$ be a connected, locally connected and semi-locally simply
  connected topological space. As for any space, we have the left
  exact reflection
  \begin{equation*}
    \cd{
      {\cat{Sh}(X)} \ar@{ (->}@<-4.5pt>[r]_-{} \ar@{}[r]|-{\bot} &
      {[\O(X)^\mathrm{op}, \cat{Set}]} \ar@<-4.5pt>[l]_-{}
    }
  \end{equation*}
  of presheaves into sheaves. Now let $U$ be a universal covering
  space for $X$, seen as an object in $\cat{Sh}(X)$, let
  $\pi_1(X) = \cat{Sh}(X)(U,U)$ be the fundamental group, and let
  $j \colon \pi_1(X) \rightarrow \cat{Sh}(X)$ be the inclusion of the
  full subcategory on $U$. By standard properties of covering spaces,
  the cocontinuous extension
  $j_! \colon \pi_1(X)\text-\cat{Set} \rightarrow \cat{Sh}(X)$ of $j$
  is fully faithful and has as essential image the covering spaces
  over $X$. In particular, we have a coreflection
  \begin{equation*}
    \cd{
      {\pi_1(X)\text-\cat{Set}} \ar@{ (->}@<-4.5pt>[r]_-{} \ar@{}[r]|-{\top} &
      {\cat{Sh}(X)} \ar@<-4.5pt>[l]_-{}
    }
  \end{equation*}
  with right adjoint sending a sheaf $S$ to the set $\cat{Sh}(X)(U,S)$
  with $\pi_1(X)$-action induced from $U$. So applying
  Proposition~\ref{prop:15}, we have a model structure on
  $[\O(X)^\mathrm{op}, \cat{Set}]$ whose fibrant objects are sheaves
  on $X$, and whose cofibrant--fibrant objects are $\pi_1(X)$-sets,
  identified with the corresponding covering spaces. General
  cofibrant objects are presheaves whose sheaf of local sections
  is a covering space.
\end{Ex}

\begin{Ex}
  \label{ex:6}
  The preceding example arose by colocalising the model structure for
  sheaves on $[\O(X)^\mathrm{op}, \cat{Set}]$ at the single object $U$
  given by the universal covering space; however, if $X$ is not
  locally semi-locally simply connected, then $U$ need not exist.
  However, we can instead take the colocalisation at the set $K$ of
  all (isomorphism-class representatives) of \emph{finite} covering
  spaces; we then obtain a model structure on
  $[\O(X)^\mathrm{op}, \cat{Set}]$ with sheaves as fibrant objects,
  and cofibrant--fibrant objects the continuous $G$-sets for $G$ the
  profinite completion of $\pi_1(X)$.
\end{Ex}

\begin{Ex}
  \label{ex:4}
  Generalising Example~\ref{ex:3} in a different direction, we can
  construct a model structure on the category
  $[\O(X)^\mathrm{op}, \cat{Vect}_k]$ of presheaves of $k$-vector
  spaces on the connected, locally connected and semi-locally simply
  connected $X$ whose fibrant objects are the sheaves of $k$-vector
  spaces and whose category of cofibrant--fibrant objects is the
  category of $k$-linear representations of $\pi_1(X)$, with these
  being identified in $[\O(X)^\mathrm{op}, \cat{Vect}_k]$ with the
  corresponding local systems.
\end{Ex}

\begin{Ex}
  \label{ex:5}
  Let $X$ be a quasi-compact quasi-separated scheme, and let
  $\cat{Psh}(\O_X)$ and $\cat{Sh}(\O_X)$ be the categories of
  presheaves of $\O_X$-modules and sheaves of $\O_X$-modules. The left
  exact reflection between sheaves and presheaves induces a left exact
  reflection $\cat{Sh}(\O_X) \leftrightarrows \cat{Psh}(\O_X)$.
  Furthermore, the subcategory
  $\cat{QCoh}(\O_X) \subseteq \cat{Sh}(\O_X)$ of \emph{quasicoherent}
  sheaves of $\O_X$-modules is coreflective
  by~\cite[Lemma~II.3.2]{Berthelot1971Theorie}. We thus have a model
  structure on the category of presheaves of $\O_X$-modules whose
  fibrant objects are the sheaves of $\O_X$-modules, and whose
  cofibrant--fibrant objects are the quasicoherent sheaves.
\end{Ex}

\begin{Ex}
  \label{ex:9}
  Recall that, if $G$ is a topological group, then a \emph{continuous}
  $G$-set is a set $X$ endowed with an action $G \times X \rightarrow
  X$ which is continuous for the discrete topology on $X$; this is
  equally the condition that the stabiliser of each $x \in X$ is an
  open subgroup of $G$. It follows easily that there is a coreflection
  \begin{equation*}
    \cd{
      {\cat{Cts}\text-G\text-\cat{Set}} \ar@{ (->}@<-4.5pt>[r]_-{} \ar@{}[r]|-{\top} &
      {G\text-\cat{Set}} \ar@<-4.5pt>[l]_-{}
    }
  \end{equation*}
  between $G$-sets and continuous $G$-sets, where the right adjoint
  $c$ sends a $G$-set $X$ to the sub-$G$-set
  $cX = \{x \in X : \mathrm{Stab}_x \text{ is open in $G$}\}$. The
  counit map is, of course, simply the inclusion, and it follows
  easily from this description that the coreflector preserves pushouts
  along maps between continuous $G$-sets; so this adjunction is
  semi-right-exact.

  Now let $N$ be an open normal subgroup of $G$. The category of
  continuous $G/N$-sets can be identified with the full subcategory of
  continuous $G$-sets in which each element is stabilised by (at
  least) $N$, and in fact we have a reflection
  \begin{equation*}
    \cd{
      {\cat{Cts}\text-G/N\text-\cat{Set}} \ar@{ (->}@<-4.5pt>[r]_-{} \ar@{}[r]|-{\bot} &
      {\cat{Cts}\text-G\text-\cat{Set}} \ar@<-4.5pt>[l]_-{}
    }
  \end{equation*}
  where the left adjoint quotients out a continuous $G$-set by the
  equivalence relation $x \sim x'$ iff $Nx = Nx'$. We thus have a
  model structure on $G\text-\cat{Set}$ whose cofibrant objects are
  the continuous $G$-sets and whose cofibrant--fibrant objects are the
  continuous $G/N$-sets. The general fibrant objects are those
  $G$-sets in which every element with an open stabiliser is
  stabilised by at least $N$.
\end{Ex}

\begin{Ex}
  \label{ex:11}
  Let $\Delta_3$ denote the full subcategory of $\Delta$ on $[0]$,
  \dots, $[3]$, and let
  $\cat{sSet}_3 = [\Delta_3^\mathrm{op}, \cat{Set}]$. Left Kan
  extension, restriction and right Kan extension along the inclusion
  $\Delta_3 \subseteq \Delta$ gives a chain of adjoints
  $\mathrm{sk}_3 \dashv \mathrm{tr}_3 \dashv \mathrm{cosk}_3 \colon
  \cat{sSet}_3 \rightarrow \cat{sSet}$ with both $\mathrm{sk}_3$ and
  $\mathrm{cosk}_3$ fully faithful. In particular,
  $\mathrm{sk}_3 \colon \cat{sSet}_3 \leftrightarrows \cat{sSet}
  \colon \mathrm{tr}_3$ is a semi-right-exact coreflection. Now, as
  the data and axioms for a category only involve at most three
  composable arrows, the \emph{truncated} nerve
  $\mathrm{tr}_3 N \colon \cat{Cat} \rightarrow \cat{sSet} \rightarrow
  \cat{sSet}_3$ is still fully faithful, and has a left adjoint $L$
  since $N$ and $\mathrm{tr}_3$ do. So we also have a reflection
  $\mathrm{tr}_3 N \colon \cat{Cat} \leftrightarrows \cat{sSet}_3
  \colon L$.
  
  So by the dual of Proposition~\ref{prop:15}, we have a model
  structure on $\cat{sSet}$ whose cofibrant objects are the
  $3$-truncated simplicial sets, and whose subcategory of
  fibrant--cofibrant objects is equivalent to $\cat{Cat}$. However,
  this equivalence does \emph{not} identify a category in the usual way
  with its nerve, but rather with the $3$-skeleton of its nerve.
  Indeed, the cofibrant--fibrant objects are simplicial sets $X$ which
  are $3$-truncated and satisfy the restricted Segal condition that
  the spine projections
  \begin{equation}\label{eq:8}
    X_2 \rightarrow X_1 \times_{X_0} X_1 \qquad  \text{and} \qquad 
    X_3 \rightarrow X_1 \times_{X_0} X_1
    \times_{X_0} X_1
  \end{equation}
  are isomorphisms: in other words, the $3$-skeleta of nerves of
  categories. More generally, the fibrant objects of this model
  structure are simplicial sets $X$, not necessarily $3$-truncated,
  for which the Segal maps in~\eqref{eq:8} are invertible.
\end{Ex}

\begin{Ex}
  \label{ex:8}
  Let $\E$ denote the category of small, strictly symmetric, strictly
  monoidal categories enriched over abelian groups. There is a full
  embedding of the category of commutative monoids into $\E$ as
  \emph{discrete} categories, and this has a right adjoint given by
  taking the set of objects. This right adjoint is clearly
  cocontinuous, and so we have a semi-right-exact coreflection
  \begin{equation*}
    \cd{
      {\cat{CMon}} \ar@{ (->}@<-4.5pt>[r]_-{} \ar@{}[r]|-{\top} &
      {\E} \ar@<-4.5pt>[l]_-{}
    }
  \end{equation*}
  On the other hand, we have the well-known construction of the
  Grothendieck group of a commutative monoid, giving a reflection
  \begin{equation*}
    \cd{
      {\cat{Ab}} \ar@{ (->}@<-4.5pt>[r]_-{} \ar@{}[r]|-{\bot} &
      {\cat{CMon}\rlap{ .}} \ar@<-4.5pt>[l]_-{}
    }
  \end{equation*}
  We therefore have a model structure on $\E$ whose cofibrant objects
  are commutative monoids and whose cofibrant--fibrant objects are
  abelian groups. The fibrant objects are the small, strictly
  symmetric, strictly monoidal $\cat{Ab}$-categories
  $(\C, \otimes, I)$ in which every object is \emph{strictly}
  invertible for the tensor product $\otimes$. Such categories $\C$
  with abelian group of objects $M$ can be identified\footnote{The
    second author
    learnt of this correspondence from James Dolan.} with $M$-graded
  commutative rings $C$, via the correspondence
  \begin{equation*}
    \C(x,y) \qquad \leftrightarrow \qquad C_{y \otimes x^{-1}}\rlap{ .}
  \end{equation*}
\end{Ex}

\begin{Ex}
  \label{ex:10}
  Let $\T$ be any finitary algebraic theory, such as the theory of
  monoids, or groups, or rings, or $k$-vector spaces, and so on. In
  each case, there is a category with finite products
  $\mathbb{T}$---the \emph{Lawvere
    theory}~\cite{Lawvere1963Functorial} associated to $\T$---whose
  objects are the natural numbers, and for which
  finite-product-preserving functors $\mathbb{T} \rightarrow \E$ into
  any category with finite products are equivalent to $\T$-models in $\E$.

  Now consider any one of the semi-right-exact coreflections
  $i \colon \A \leftrightarrows \E \colon Q$ from the above examples
  (i.e., from Examples~\ref{ex:7}, \ref{ex:9}, \ref{ex:11} or
  \ref{ex:8}). Postcomposition with $Q$ and $i$ induces a
  semi-right-exact coreflection on functor categories
  \begin{equation*}
    \cd{
      {\A^\mathbb{T}} \ar@{ (->}@<-4.5pt>[r]_-{i^\mathbb{T}} \ar@{}[r]|-{\top} &
      {\E^\mathbb{T}} \ar@<-4.5pt>[l]_-{Q^\mathbb{T}}
    }
  \end{equation*}
On the other hand, the category $\cat{FP}(\mathbb{T},\A)$ of
finite-product-preserving functors $\mathbb{T} \rightarrow \A$ is 
reflective in $\A^\mathbb{T}$; and so, applying
the dual of Proposition~\ref{prop:15}, we obtain a model structure on
$\E^\mathbb{T}$ whose cofibrant objects are functors $\mathbb{T}
\rightarrow \A$ and whose fibrant--cofibrant objects are
$\T$-models in $\A$. The general fibrant objects are functors
$\mathbb{T} \rightarrow \E$ whose postcomposition with $Q \colon \E
\rightarrow \A$ preserves finite products; these are equally those
functors $\mathbb{T} \rightarrow \E$ which preserve finite products
\emph{up to} a map which is inverted by $Q$.
\end{Ex}

Let us conclude the paper by taking stock of the preceding examples.
As we explained in the introduction, this paper is really a prelude to
further work explaining Bousfield (co)localisation of general model
structures in terms of projective and injective liftings; in which
context, of course, many serious examples already exist. In the
one-dimensional context, there were no non-trivial examples at all,
and so we felt compelled to provide some.

These examples are sufficiently natural that their existence is surely
not without force. However, it is as yet unclear to us what this force
may be. One possibility is that some aspects of model category theory
do not trivialise in this setting, and provide interesting
information. This may be true, for example, for the theory of monoidal
model categories, or the theory of homotopy limits and colimits.

More interestingly, it could be that the existence of a
one-dimensional model structure is alerting us to some as-yet
unconsidered good property of the interaction between a reflective and
a coreflective subcategory. It is suggestive that many of the
\emph{right proper} examples we have found come from topos theory and,
in particular, relate either to the Galois theory of
Grothendieck~\cite{1971Revetements} or to the \emph{petit}
topos--\emph{gros} topos dichotomy explored in, for example,
\cite{Dubuc1986Logical, Johnstone2012Calibrated}. It would be very
interesting to see if these links can be made tighter. To this end,
some natural questions to explore would be: can we give a
topos-theoretic characterisation of all right proper one-dimensional
model structures on a presheaf category $[\C^\mathrm{op}, \cat{Set}]$?
What information do Quillen adjunctions between such model structures
carry?

\end{document}